\documentclass[reqno, 11pt]{amsart}

\usepackage{tikz}
\usepackage{mathrsfs}
\usepackage{amsthm}
\usepackage{amsmath,amssymb}
\usepackage[all,poly,knot]{xy}
\usepackage{hyperref}
\usepackage{amsfonts}
\usepackage{color} 
\usepackage{diagbox}
\usepackage{chemfig} 
\usepackage{tikz}
\usepackage[square,sort,comma,numbers]{natbib}

\newtheorem{thm}{Theorem}[section]

\newtheorem{cor-defi}[thm]{{Corollary-Definition}}
\newtheorem{lem}[thm]{{Lemma}}

\newtheorem{defi}[thm]{Definition}
\theoremstyle{remark}

\newcommand{\bD}{{\mathbb  D}}
\newcommand{\bE}{{\mathbb  E}}
\newcommand{\bF}{{\mathbb  F}}

\newcommand{\bZ}{{\mathbb  Z}}

\newcommand{\mB}{{\mathcal  B}}

\newcommand{\mD}{{\mathcal  D}}
\newcommand{\mE}{{\mathcal  E}}

\newcommand{\mL}{{\mathcal  L}}

\newcommand{\mO}{{\mathcal  O}}

\newcommand{\mU}{{\mathcal  U}}

\newcommand{\mX}{{\mathcal  X}}

\newcommand{\sU}{{\mathscr  U}}

\newcommand{\sX}{{\mathscr  X}}

\newcommand{\nc}{\newcommand}
\nc{\on}{\operatorname}

\nc{\Aut}  {{\on{\mathrm  {Aut}}}}
\nc{\End}  {{\on{\mathrm  {End}}}}
\nc{\Fil}  {{\on{\mathrm  {Fil}}}}
\nc{\Frac} {{\on{\mathrm  {Frac}}}}
\nc{\Gal}  {{\on{\mathrm  {Gal}}}}
\nc{\GL}   {{\on{\mathrm  {GL}}}}
\nc{\Gr}   {{\on{\mathrm  {Gr}}}}
\nc{\Hom}  {{\on{\mathrm  {Hom}}}}
\nc{\id}   {{\on{\mathrm  {id}}}}
\nc{\PGL}  {{\on{\mathrm  {PGL}}}}
\nc{\rank} {{\on{\mathrm  {rank}}}}
\nc{\rmd}  {{\on{\mathrm  {d}}}}
\nc{\Spec} {{\on{\mathrm  {Spec}}}}

\def\can {{\mathrm  {can}}}

\nc{\HDF}  {{\on{\mathcal  {HDF}}}}
\nc{\HIG}  {{\on{\mathcal  {HIG}}}}
\nc{\IC}   {{\on{\mathcal {IC}}}}
\nc{\MCF}  {{\on{\mathcal {MCF}}}}
\nc{\MCFa} {{\on{\mathcal {MCF}_{[0,a]}}}}
\nc{\MF}   {{\on{\mathcal {MF}}}}
\nc{\MFa}  {{\on{         \MF_{[0,a]}}}}
\nc{\MFaf} {{\on{         \MF_{[0,a],f}}}}
\nc{\MIC}  {{\on{\mathcal {MIC}}}} 
\nc{\MICa} {{\on{\mathcal {MIC}_{[0,a]}}}} 
\nc{\THDF} {{\on{\mathcal {THDF}}}}
\nc{\THDFa}{{\on{\mathcal {THDF}_{[0,a]}}}}
\nc{\TMF}  {{\on{\mathcal {TMF}}}}
\nc{\TMFa} {{\on{\TMF_{[0,a]}}}}
\nc{\TMFaf}{{\on{\TMF_{[0,a],f}}}}
\nc{\tMIC} {{\on{\widetilde{\mathcal{MIC}}}}}

\def\hR{{\widehat{R}}} 
\def\tmB{{\widetilde{\mB}}}
\def\mBR{{\mB_{R_\pi}}} 
\def\tmBR{{\widetilde{\mB}_{R_\pi}}}
\def\tnabla{{\widetilde{\nabla}}}
\def\tV{{\widetilde V}}
\def\tsX{{\widetilde{\sX}}}

\def\R{R}

\begin{document}

	\title[Base change of Twisted Higgs-de Rham flows over very ramified valuation rings]{Base change of twisted Fontaine-Faltings modules and Twisted Higgs-de Rham flows over very ramified valuation rings} 
	\author{Ruiran Sun}
	\author{Jinbang Yang}
	\author{Kang Zuo}
	
	\email{ruirasun@uni-mainz.de, yjb@mail.ustc.edu.cn, zuok@uni-mainz.de} 
	\address{Institut F\"ur Mathematic, Universit\"at Mainz, Mainz, 55099, Germany}
\thanks{This work was supported by SFB/Transregio 45 Periods, Moduli Spaces and Arithmetic of Algebraic
	Varieties of the DFG (Deutsche Forschungsgemeinschaft)  and also supported by  National Key Basic Research Program of China (Grant No. 2013CB834202).}  
	
	\date{}

	\maketitle
	
\begin{abstract}
In this short notes, we prove a stronger version of Theorem $0.6$ in our previous paper \cite{SYZ17}:
Given a smooth log scheme $(\mathcal{X} \supset \mathcal{D})_{W(\mathbb{F}_q)}$,  each stable twisted $f$-periodic logarithmic Higgs bundle $(E,\theta)$ over the closed fiber $(X \supset D)_{\mathbb{F}_q}$ will correspond to a $\mathrm{PGL}_r(\mathbb{F}_{p^f})$-crystalline representation of $\pi_1((\mathcal{X} \setminus \mathcal{D})_{W(\mathbb{F}_q)[\frac{1}{p}]})$ such that its restriction to the geometric fundamental group is absolutely irreducible.
\end{abstract}

\section{Introduction}

In the previous paper \cite{SYZ17}, we prove the following results
\begin{thm}[Theorem $0.6$ in \cite{SYZ17}]
Let $k$ be a finite field of characteristic $p$. Let $\mX$ be a smooth proper geometrically connected scheme over $W(k)$ together with a smooth log structure $\mD/W(k)$. Assume that there exists a semistable graded logarithmic Higgs bundle 
		$(E,\theta)/(\mX,\mD)_1$  with $r:=\mathrm{rank}( E) \leq p-1,$  discriminant $\Delta_H(E)=0$, $r$ and $\deg_H(E)$ are coprime. Let $\mX^o=\mX \setminus \mD$ and $K'=W(k\cdot\mathbb{F}_{p^f})[1/p]$. 	 Then, there exist a positive integer $f$ and a $\mathrm{PGL}_r(\mathbb{F}_{p^f})$-crystalline representation $\rho$ of $\pi_1(\mX^o_{K'})$, 
		which is  irreducible in $\mathrm{PGL}_r(\overline{\mathbb{F}}_p)$.  
\end{thm}
In the proof we use the so called \emph{twisted periodic Higgs-de Rham flow}, which is a variant of the periodic Higgs-de Rham flow defined by Lan, Sheng and Zuo in \cite{LSZ13a}.\\ 

Now we want to improve this result to a stronger version. The $\mathrm{PGL}_r(\mathbb{F}_{p^f})$-crystalline representation of $\pi_1((\mathcal{X} \setminus \mathcal{D})_{W(\mathbb{F}_q)[\frac{1}{p}]})$ corresponding to the stable Higgs bundle should have the following property: its restriction to the geometric fundamental group $\pi_1((\mathcal{X} \setminus \mathcal{D})_{\bar{\mathbb{Q}}_p})$ is absolutely irreducible. Here the "absolutely irreducible" means that the representation is still irreducible after extending the coefficient $\mathbb F_{p^f}$ to $\overline{\mathbb F}_p$.

We outline the proof as follows. Fix a $K_0$-point in $X_{K_0}$, one can pull back the representation $\rho$ to a representation of the galois group, whose image is finite. This finite quotient will give us a field extension $K/K_0$ such that the restriction of $\rho$ on $\mathrm{Gal}(\bar{K_0}/K)$ is trivial. That means $\rho(\pi_1(X_{K}^o))=\rho(\pi_1(X^o_{\overline{K}_0}))$. So it suffices to prove the irreducibility of $\rho$ on $\pi_1(X_{K}^o)$, which gives us the chance to apply the method of twisted periodic Higgs-de Rham flows as in \cite{SYZ17}. But the field extension $K/K_0$ is usually ramified. So we have to work out the construction in \cite{SYZ17} to the very ramified case.

\section{Base change of twisted Fontaine-Faltings modules and Twisted Higgs-de Rham flows over very ramified valuation rings}

\subsection{Notations in the case of $\mathrm{Spec}k$.} In this notes, $k$ will always be a perfect field of characteristic $p>0$. Let $\pi$ be a root of an Eisenstein polynomial
\[f(T)=T^e+\sum_{i=0}^{e-1}a_i T^i\]
of degree $e$ over the Witt ring $W=W(k)$. Denote $K_0=\Frac(W)=W[\frac1p]$ and $K=K_0[\pi]$, where $K_0[\pi]$ is a totally ramified extension of $K_0$ of degree $e$. Denote by $W_\pi=W[\pi]$ the ring of integers of $K$, which is a complete discrete valuation ring with maximal ideal $\pi W_\pi$ and the residue field $W_\pi/\pi W_\pi=k$. Denote by $W[[T]]$ the ring of formal power-series over $W$. Then 
\[W_\pi=W[[T]]/fW[[T]].\]
The PD-hull $\mB_{W_\pi}$ of $W_\pi$ is the PD-completion of the ring obtained by adjoining to $W[[T]]$ the divided powers $\frac{f^n}{n!}$. 
 More precisely
\[\mB_{W_\pi}=\left\{\left.\sum_{n=0}^{\infty}a_nT^n\in K_0[[T]] \right| a_n[n/e]!\in W \text{ and }a_n[n/e]!\rightarrow 0\right\}.\]
A decreasing filtration is defined on $\mB_{W_\pi}$ by the rule that $F^q(\mB_{W_\pi})$ is the closure of the ideal generated by divided powers $\frac{f^n}{n!}$ with $n\geq q$. Note that the ring $\mB_{W_\pi}$ only depends on the degree $e$ while this filtration depends on $W_\pi$ and $e$. One has 
\[\mB_{W_\pi}/\Fil^1\mB_{W_\pi}\simeq W_\pi. \]
There is a unique continuous homomorphism of $W$-algebra $\mB_{W_\pi}\rightarrow B^+(W_\pi)$ which sends $T$ to $[\underline{\pi}]$. Here $\underline{\pi} = (\pi,\pi^{\frac{1}{p}},\pi^{\frac{1}{p^2}},\dots) \in \varprojlim \bar{R}$. 
We denote
\[\tmB_{W_\pi}=\mB_{W_\pi}[\frac{f}{p}]\]
which is a subring of $K_0[[T]]$. The idea $(\frac{f}{p})$ induces a decreasing filtration $\Fil^\cdot \tmB_{W_\pi}$ such that
\[\tmB_{W_\pi}/\Fil^1\tmB_{W_\pi}\simeq W_\pi. \]
The Frobenius endomorphism on $W$ can be extended to an endomorphism $\varphi$ on $K_0[[T]]$, where $\varphi$ is given by $\varphi(T)=T^p$. Since $\varphi(f)$ is divided by $p$, we have $\varphi(\tmB_{W_\pi})\subset \mB_{W_\pi}$. Thus one gets two restrictions 
\[\varphi:\tmB_{W_\pi}\rightarrow \mB_{W_\pi} \text{ and } \varphi:\mB_{W_\pi}\rightarrow \mB_{W_\pi} .\]
Note that the ideal of $\mB_{W_\pi}$, generated by $\Fil^1\mB_{W_\pi}$ and $T$, is stable under $\varphi$. Then we have the following commutative diagram
\begin{equation}\label{diag:FrobLift}
\xymatrix{
\mB_{W_\pi}\ar[d]^{\varphi} \ar@{->>}[r]  &\mB_{W_\pi}/(\Fil^1\mB_{W_\pi},T)=k \ar[d]^{(\cdot)^p}\\
\mB_{W_\pi} \ar@{->>}[r] &\mB_{W_\pi}/(\Fil^1\mB_{W_\pi},T)=k \\
}
\end{equation}
Since $\varphi(\frac{f}{p})$ is invertible in $\mB_{W_\pi}$, we have $\varphi((\Fil^1\tmB_{W_\pi},T))\not\subset(\Fil^1\mB_{W_\pi},T)$. Hence the following diagram is not commutative
\begin{equation}\label{diag:FrobLift}
\xymatrix{
\tmB_{W_\pi}\ar[d]^{\varphi} \ar@{->>}[r]  &\tmB_{W_\pi}/(\Fil^1\tmB_{W_\pi},T)=k \ar[d]^{(\cdot)^p}\\
\mB_{W_\pi} \ar@{->>}[r] &\mB_{W_\pi}/(\Fil^1\mB_{W_\pi},T)=k\\
}
\end{equation}
\subsection{Base change in the small affine case.}

For a smooth and small $W$-algebra $\R$ (which means there exists an \'etale map  
	\[W[T_1^{\pm1},T_2^{\pm1},\cdots, T_{d}^{\pm1}]\rightarrow R,\]
 see \cite{Fal89}),  Lan-Sheng-Zuo constructed categories $\MIC(\R/p\R)$, $\tMIC(\R/p\R)$, $\MCF(\R/p\R)$ and $\MF(\R/p\R)$. 
A Fontaine-Faltings module over $\R/p\R$ is an object $(V,\nabla,\Fil)$ in $\MCF(\R/p\R)$ together with an isomorphism $\varphi: \widetilde{(V,\nabla,\Fil)}\otimes_{\Phi}\hR\rightarrow (V,\nabla)$ in $\MIC(\R/p\R)$, where $\widetilde{(\cdot)}$ is the Faltings' tilde functor.

We generalize those categories over the $W_\pi$-algebra $R_\pi=R\otimes_W{W_\pi}$. In general, there does not exist Frobenius lifting on the p-adic completion of $\hR_\pi$. We lift the absolute Frobenius map on $R_\pi/\pi R_\pi$ to a map $\Phi:\mBR\rightarrow \mBR$
\begin{equation}\label{diag:FrobLift}
\xymatrix{
 \mBR  \ar[d]^{\Phi} \ar@{->>}[r]  &  R_\pi/\pi R_\pi=\R/p\R \ar[d]^{(\cdot)^p}\\
\mBR \ar@{->>}[r]  &  R_\pi/\pi R_\pi=\R/p\R\\
}
\end{equation}
where $\mBR$ is the $p$-adic completion of $\mB_{W_\pi}\otimes_W \R$. This lifting is compatible with $\varphi:\mB_{W_\pi}\rightarrow \mB_{W_\pi}$.
Denote $\tmBR=\mBR[\frac{f}{p}]$. Then $\Phi$ can be extended to 
\[\Phi:\tmBR\rightarrow \mBR\]
uniquely, which is compatible with $\varphi:\tmB_{W_\pi}\rightarrow \mB_{W_\pi}$. 
The filtrations on $\mB_{W_\pi}$ and $\tmB_{W_\pi}$ induce filtrations on $\mBR$ and $\tmBR$ respectively, which satisfy
\[\mBR/\Fil^1\mBR\simeq \hR_\pi\simeq \tmBR/\Fil^1\tmBR.\]
\begin{lem}\label{lem:F^iB&F^itB}
Let $n<p$ be a natural number and let $b$ be an element in $F^n\mBR$. Then $\frac{b}{p^n}$ is an element in $F^n\tmBR$.
\end{lem}
\begin{proof} Since the filtrations on $\mBR$ and $\tmBR$ are induced by those on  $\mB_{W_\pi}$ and $\tmB_{W_\pi}$ respectively, we have 
\[F^n\mBR=\left\{\left.\sum_{i\geq n}^{\infty}a_i\frac{f^i}{i!}\right| a_i\in \hR[[T]] \text{ and }a_i \rightarrow 0\right\},\]
and
\begin{equation}
\begin{split}
F^n\tmBR &=\left\{\left.\sum_{i\geq n}^{\infty}a_i\frac{f^i}{p^i}\right| a_i\in \mBR \text{ and }a_i=0 \text{ for } i\gg0\right\}\\
&=\left\{\left.\sum_{i\geq n}^{\infty}a_i\frac{f^i}{p^i}\right| a_i\in \hR[[T]] \text{ and }\frac{i!}{p^i}a_i\rightarrow 0 \right\}.\\
\end{split}
\end{equation} 
Assume $b=\sum_{i\geq n}a_i\frac{f^i}{i!}$ with $a_i\in \hR[[T]]$ and $a_i\rightarrow 0$, then 
\[\frac{b}{p^n}=\sum_{i\geq n} \frac{p^i a_i }{p^n i!}\cdot \frac{f^i}{p^i},\] 
and the lemma follows.
\end{proof}

\begin{lem}
There are isomorphisms of free $\hR_\pi$-modules of rank $1$   
\begin{equation}
\xymatrix@R=0mm{
\Gr^n \tmBR \ar[r]&  \Gr^n \mBR \ar[r] & \hR_\pi  \\
 \frac{f^n}{p^n}\ar@{|->}[r] & \frac{f^n}{n!}\ar@{|->}[r] & 1 
}
\end{equation}
\end{lem}
\begin{proof}This can be checked directly.
\end{proof}

Recall that $\mBR$ and $\tmBR$ are  $\hR$-subalgebras of $\hR(\frac{1}{p})[[T]]$. We denote by $\Omega^1_{\mBR}$ (resp. $\Omega^1_{\tmBR}$) the $\mBR$-submodule (resp. $\tmBR$-submodule) of $\Omega^1_{\hR(\frac{1}{p})[[T]]/W}$ generated by elements $\rmd b$, where $b\in \mBR$ (resp. $b\in \tmBR$).  There is a filtration on $\Omega^1_{\mBR}$ (resp. $\Omega^1_{\tmBR}$) given by
$F^n \Omega^1_{\mBR}=F^n \mBR \cdot \Omega^1_{\mBR}$ (resp. $F^n \Omega^1_{\tmBR}=F^n \tmBR \cdot \Omega^1_{\tmBR}$). One gets the following result directly by Lemma~\ref{lem:F^iB&F^itB}. 
\begin{lem}
Let $n<p$ be a natural number. Then $\frac{1}{p^n}F^n \Omega^1_{\mBR}\subset F^n \Omega^1_{\tmBR}$.
\end{lem} 

We have the following categories 
\begin{itemize}
\item $\MIC(\mBR/p\mBR)$: the category of free $\mBR/p\mBR$-modules with integrable connections.
\item $\tMIC(\tmBR/p\tmBR)$: the category of free $\tmBR/p\tmBR$-modules with integrable nilpotent $p$-connection.
\item $\MCFa(\mBR/p\mBR)$: the category of filtered free $\mBR/p\mBR$-modules equipped with integrable connections, which satisfy the Griffiths transversality, and each of these $\mBR/p\mBR$-modules admits a filtered basis $v_i$ of degree $q_i$, $0\leq q_i\leq a$.
\end{itemize} 
A Fontaine-Faltings module over $\mBR/p\mBR$ of weight $a$ ($0\leq a\leq p-2$) is an object $(V,\nabla,\Fil)$ in the category $\MCFa(\mBR/p\mBR)$ together with an isomorphism in $\MIC(\mBR/p\mBR)$
\[\varphi: \widetilde{(V,\nabla,\Fil)}\otimes_{\Phi}\mBR\rightarrow (V,\nabla),\] 
where $\widetilde{(\cdot)}:\MCF(\mBR/p\mBR) \rightarrow \tMIC(\tmBR/p\tmBR)$ is an analogue of the Faltings' tilde functor.
For an object $(V,\nabla,\Fil)$ in $\MCF(\mBR/p\mBR)$ with filtered basis $v_i$ (of degree $q_i$, $0\leq q_i\leq a$), $\tV$ is defined as a filtered free $\tmBR/p\tmBR$-module 
\[\tV=\bigoplus_i \tmBR/p\tmBR\cdot [v_i]_{q_i}\]
with filtered basis $[v_i]_{q_i}$ (of degree $q_i$, $0\leq q_i\leq a$). Informally one can view $[v_i]_{q_i}$ as ``$\frac{v_i}{p^{q_i}}$". 
Since $\nabla$ satisfies the Griffiths transversality, there are $\omega_{ij}\in F^{q_j-1-q_i}\Omega^1_\mBR$ satisfying
\[\nabla(v_j)=\sum_{i} v_i\otimes \omega_{ij}.\] 
Since $q_j-1-q_i<a\leq p-2$, $\frac{\omega_{ij}}{p^{q_j-1-q_i}}\in F^{q_j-1-q_i}\Omega^1_\tmBR$.
We define a $p$-connection $\tnabla$ on $\tV$  via 
\[\tnabla([v_j]_{q_j})=\sum_{i} [v_i]_{q_i}\otimes \frac{\omega_{ij}}{p^{q_j-1-q_i}}.\]
\begin{lem}
The $\tmBR/p\tmBR$-module $\tV$ equipped with the $p$-connection $\tnabla$ is independent of the choice of the filtered basis $v_i$ up to a canonical isomorphism.   
\end{lem}

\begin{proof}
We write  $v=(v_1,v_2,\cdots)$ and $\omega=(\omega_{ij})_{i,j}$. Then
\[\nabla(v)=v\otimes \omega \text{ and } \tnabla([v])=[v]\otimes (pQ\omega Q^{-1}),\] 
where $Q=\mathrm{diag}(p^{q_1},p^{q_2},\cdots)$ is a diagonal matrix.
Assume that $v_i'$ is another filtered basis (of degree $q_i$, $0\leq q_i\leq a$) and $(\tV',\tnabla')$ is the corresponding $\tmBR/p\tmBR$-module equipped with the $p$-connection. Similarly, we have
\[\nabla(v')=v'\otimes \omega' \text{ and } \tnabla([v'])=[v']\otimes (pQ\omega' Q^{-1}),\] 
Assume $v_j'=\sum_i a_{ij}v_i$ ($a_{ij}\in F^{q_j-q_i}\mBR$). Then $A=(a_{ij})_{i,j}\in \GL_{\rank(V)}(\mBR)$ and 
$QAQ^{-1}=\left(\frac{a_{ij}}{p^{q_j-q_i}}\right)_{i,j}\in \GL_{\rank(V)}(\tmBR)$. 
We construct an isomorphism from $\tV'$ to $\tV$ by
\[\tau([v'])=[v] \cdot (QAQ^{-1}),\]
where $[v]=([v_1]_{q_1},[v_2]_{q_2},\cdots)$ and $[v']=([v_1']_{q_1},[v_2']_{q_2},\cdots)$. Now we only need to check that $\tau$ preserve the $p$-connections. Indeed, 
\begin{equation}
\tau\circ\tnabla'([v']) =[v]\otimes (QAQ^{-1}\cdot p Q\omega' Q^{-1})=[v]\otimes (pQ\cdot A\omega'\cdot Q^{-1})
\end{equation}
and 
\begin{equation}
\begin{split}
\tnabla\circ\tau([v']) & =\tnabla([v]QAQ^{-1})\\
& =[v]\otimes (pQ\omega Q^{-1}\cdot QAQ^{-1} +p\cdot Q \rmd A Q^{-1})\\
&=[v]\otimes (pQ\cdot (\omega A  + \rmd A)\cdot Q^{-1})\\
\end{split}
\end{equation}
Since $\nabla(v')=\nabla(vA)=v\otimes \rmd A + v\otimes \omega A=v'\otimes (A^{-1}\rmd A+ A^{-1}\omega A)$, we have $\omega'=A^{-1}\rmd A+ A^{-1}\omega A$ by definition. Thus $\tau\circ\tnabla'=\tnabla\circ\tau$.
\end{proof}
The functor 
\[-\otimes_\Phi\mBR : \tMIC(\tmBR/p\tmBR) \rightarrow \MIC(\mBR/p\mBR)\]
 is induced by base change under $\Phi$. Note that the connection on $(\tV,\tnabla)\otimes_\Phi\mBR$ is given by
\[\rmd + \frac{\rmd \Phi}{p}(\Phi^*\tnabla)\]
Denote by $\MFa(\mBR/p\mBR)$ the category of all Fontaine-Faltings module over $\mBR/p\mBR$ of weight $a$. 

\begin{lem}\label{lem:coeffExt} We have the following commutative diagram by extending the coefficient ring from $\R$ to $\mBR$ (or $\tmBR$)
\begin{equation*}
\xymatrix@C=1cm{
\MCF(\R/p\R) \ar[r]^{\widetilde{(\cdot)}}  \ar[d]^{-\otimes_R \mBR} 
& \tMIC(\R/p\R)\ar[r]^{-\otimes_{\Phi} {\R}}  \ar[d]^{-\otimes_R \tmBR} 
& \MIC(\R/p\R)  \ar[d]^{-\otimes_R \mBR} \\
\MCF(\mBR/p\mBR) \ar[r]^{\widetilde{(\cdot)}}  
& \tMIC(\tmBR/p\tmBR)\ar[r]^{-\otimes_\Phi\mBR}  
& \MIC(\mBR/p\mBR) \\ 
}
\end{equation*} 
In particular, we get a functor from the category of Fontaine-Faltings modules over $R/pR$ to that over $\mBR/p\mBR$  
\[\MFa(\R/p\R)\rightarrow \MFa(\mBR/p\mBR).\] 
\end{lem}
Those categories of Fontaine-Faltings modules are independent of the choice of the Frobenius lifting by the Taylor formula.
\begin{thm}\label{thm:gluingFFmod} For any two choices of $\Phi_{\mBR}$ there is an equivalence between the corresponding categories $\MFa(\mBR/p\mBR)$ with different $\Phi_{\mBR}$. These equivalences satisfy the obvious cocycle condition. Therefore, $\MFa(\mBR/p\mBR)$ is independent of the choice of $\Phi_{\mBR}$ up to a canonical isomorphism.
\end{thm}

\begin{defi}\label{lem:defiFunctorD}
For an object $(V,\nabla,\Fil,\varphi)$ in $\MFa(\mBR/p\mBR)$, denote
\[\bD(V,\nabla,\Fil,\varphi)=\Hom_{B^+(R),\Fil,\varphi}(V\otimes_{\mBR} B^+(R),B^+(R)/pB^+(R)).\]  
\end{defi}
The proof of Theorem~2.6 in~\cite{Fal89} works in this context. we can define an adjoint functor $\bE$ of $\bD$ as 
\[\bE(L)= \varinjlim \{V\in \MFa(\mBR/p\mBR)\mid L\rightarrow \bD(V)\}.\]
The proof in page~41 of~\cite{Fal89} still works. Thus we obtain:
\begin{thm} $i).$ The homomorphism set $\bD(V,\nabla,\Fil,\varphi)$ is an $\bF_p$-vector space with a linear $\mathrm{Gal}(\overline{R}_K/R_K)$-action whose rank equals to the rank of $V$.\\
$ii).$ The functor $\bD$ from $\MFa(\mBR/p\mBR)$ to the category of $W_n(\bF_p)$-$\mathrm{Gal}(\overline{R}_K/R_K)$-modules is fully faithful and its image on objects is closed under subobjects and quotients. 
\end{thm}

\subsection{ Categories and Functors on proper smooth variety over very ramified ring $W_\pi$} 
Let $\mX$ be a smooth proper $W$-scheme and $\mX_\pi=\mX\otimes_{W}W_\pi$.
Let $\sX_\pi$ be the formal completion of $\mX\otimes_{W}\mB_{W_\pi}$ and $\tsX_\pi$ be the formal completion of $\mX\otimes_W{\tmB_{W_\pi}}$. Then $\sX_\pi$ is an infinitesimal thickening of $\mX_\pi$ and the ideal defining $\mX_\pi$ in $\sX_\pi$ has a nilpotent PD-structure which is compatible with that on $F^1(\mB_{W_\pi})$ and $(p)$
\begin{equation}
\xymatrix@R=2mm{
\mX_\pi \ar[rr] \ar[dd] \ar[dr] && \sX_\pi \ar[dd]|(0.5)\hole \ar[dr]&\\
& \widetilde{\sX_\pi}  \ar[rr] \ar[dd] && \mX \ar[dd]\\
\Spec W_\pi  \ar[rr]|(0.5)\hole
 \ar[dr] && \Spec \mB_{W_\pi} \ar[dr] &\\
& \Spec\widetilde{\mB}_\pi \ar[rr] && \Spec W \ .\\
}
\end{equation}
 Let $\{\mU_{i}\}_i$ be a covering of small affine open subsets of $\mX$. By base change, we get a covering $\{\sU_i=\mU_{i}\times_\mX \sX_\pi\}_i$ of $\sX_\pi$ and a covering $\{\widetilde{\sU}_i=\mU_{i}\times_\mX \tsX_\pi\}_i$ of $\widetilde{\sX_\pi}$. For each $i$, we denote $R_{i}=\mO_{\mX_\pi}(\mU_{i}\times_\mX \mX_\pi)$. Then $\mB_{R_i}=\mathcal O_{\sX_\pi}(\sU_i)$ and 
$\tmB_{R_i}=\mathcal O_{\widetilde{\sX_\pi}}(\widetilde{\sU}_i)$ are the coordinate rings. Fix a Frobenius-lifting $\Phi_i:\tmB_{R_i}\rightarrow \mB_{R_i}$, one gets  those categories of Fontaine-Faltings modules
\[\MFa(\mB_{R_i}/p\mB_{R_i}).\]
By the Theorem~\ref{thm:gluingFFmod}, these categories are glued into one category. Moreover those underlying modules are glued into a bundle over $\sX_{\pi,1}=\sX_\pi\otimes_{\bZ_p}\bF_p$. We denote this category by $\MFa(\sX_{\pi,1})$. 

\subsubsection{Inverse Cartier functor and a description of $\MFa(\sX_{\pi,1})$ via Inverse Cartier functor} Let  $\overline{\Phi}: \mBR/p\mBR \rightarrow \mBR/p\mBR$ be the $p$-th power map. Then we get the following lemma directly.
\begin{lem}\label{lem:FrobLift}
Let $\Phi:\mBR\rightarrow \mBR$ and $\Psi:\mBR\rightarrow \mBR$ be two liftings of $\overline{\Phi}$ which are both compatible with the Frobenius map on $\mB_{W_\pi}$. 
\begin{itemize}
\item[i).] Since $\varphi(f)$ is divided by $p$, we extend $\Phi$ and $\Psi$ to maps on $\tmBR$ via $\frac{f^n}{p}\mapsto \left(\frac{\varphi(f)}{p^n}\right)^n$  uniquely;
\item[ii).] the difference $\Phi-\Psi$ on $\tmBR$ is still divided by $p$;
\item[iii).] the differentials $\rmd \Phi : \Omega^1_{\tmBR}\rightarrow \Omega^1_{\mBR}$ and $ \rmd \Psi : \Omega^1_{\tmBR}\rightarrow \Omega^1_{\mBR}$ are divided by $p$.
\end{itemize}
\end{lem} 
From now on, we call the extension given by i) of Lemma~\ref{lem:FrobLift} the \emph{Frobenius liftings} of $\overline{\Phi}$ on $\tmBR$.

\begin{lem}Let $\Phi:\tmBR\rightarrow \mBR$ and $\Psi:\tmBR\rightarrow \mBR$ be two Frobenius liftings of $\overline{\Phi}$ on $\tmBR$. Then there exists a $\mBR/p\mBR$-linear morphism \[h_{\Phi,\Psi}:\Omega^1_{\tmBR/p\tmBR}\otimes_{\overline\Phi}\mBR/p\mBR \rightarrow  \mBR/p\mBR\]
 satisfying that:
\begin{itemize}
\item[i).] we have $\frac{\rmd \Phi}{p}-\frac{\rmd \Psi}{p}=\rmd h_{\Phi,\Psi}$ over $\Omega^1_{\tmBR/p\tmBR}\otimes_{\overline\Phi}1$;
\item[ii).] the cocycle condition holds.
\end{itemize}
\end{lem}
\begin{proof}
As $\Omega^1_{\tmBR/p\tmBR}\otimes_{\overline\Phi}\mBR/p\mBR$ is an $\mBR/p\mBR$-module generated by elements of the form $\rmd g\otimes 1$ ($g\in\tmBR/p\tmBR$) with relations $\rmd(g_1+g_2)\otimes 1-\rmd g_1 \otimes 1-\rmd g_2 \otimes 1$ and $\rmd(g_1g_2)\otimes 1-\rmd g_1 \otimes\overline{\Phi}(g_2)-\rmd g_2 \otimes \overline{\Phi}(g_1)$.
Since $\Phi-\Psi$ is divided by $p$, we can denote $h_{ij}({\rmd g}\otimes 1)=\frac{\Phi(\hat{g})-\Psi(\hat{g})}{p}\pmod{p}\in \mBR/p\mBR$ for any element $g\in \mO_{U_1}$ (the definition does not depend on the choice of the lifting $\hat{g}$ of $g$ in $\mO_{\mU}$). By direct calculation, we have
\[h_{ij}(\rmd(g_1+g_2)\otimes 1)=h_{ij}(\rmd g_1 \otimes 1)+h_{ij}(\rmd g_2 \otimes 1)\]
and
\[h_{ij}(\rmd(g_1g_2)\otimes 1)=\overline{\Phi}(g_2)\cdot h_{ij}(\rmd g_1 \otimes 1)+\overline{\Phi}(g_1)\cdot h_{ij}(\rmd g_2 \otimes 1)\]
Thus $h_{ij}$ can be $\mBR/p\mBR$-linearly extended.
One checks i) and ii) directly by definition.
\end{proof}

Let $(\tV,\tnabla)$ be a locally filtered free sheaf over $\tsX_{\pi,1}=\tsX_\pi\otimes_{\bZ_p}\bF_p$ with an integrable $p$-connection. Here a ``filtered free'' module over a filtered ring $R$ is a direct sum of copies of $R$ with the filtration shifted by a constant amount. The associated graded then has a basis over $gr_F(R)$ consisting of homogeneous elements(see \cite{Fal99}).  Let $(\tV_i,\tnabla_i)=(\tV,\tnabla)\mid_{\widetilde{\sU}_{i,1}}$ be its restriction on the open subset $\widetilde{\sU}_{i,1}=\widetilde{\sU}_{i}\otimes_{\bZ_p}\bF_p$. By taking functor $\Phi_i^*$, we get bundles with integrable connections over $\sU_{i,1}=\sU_i\otimes_{\bZ_p}\bF_p$
\[\left(\Phi_i^*\tV_i, \rmd + \frac{\rmd \Phi_i}{p}(\Phi_i^*\tnabla)\right).\]
\begin{lem}\label{gluingFlatBundle} Let $(\tV,\tnabla)$ be a locally filtered free sheaf over $\tsX_{\pi,1}$ with an integrable $p$-connection. Then these local bundles with connections
	\[\left(\Phi_i^*\tV_i, \rmd + \frac{\rmd \Phi_i}{p}(\Phi_i^*\tnabla)\right)\]
can be glued into a global bundle with a connection on $\sX_{\pi,1}$ via transition functions
\[G_{ij}=\exp\left(h_{\Phi_i,\Phi_j}(\overline{\Phi}^*\tnabla)\right):\Phi_{i}^*(\tV_{ij}) \rightarrow \Phi_{j}^*(\tV_{ij}).\] 
Denote this global bundle with connection by $C^{-1}_{\sX_{\pi,1}}(\tV,\tnabla)$. Then we can construct a functor 
\[C^{-1}_{\sX_{\pi,1}}: \tMIC(\tsX_{\pi,1})\rightarrow \MIC(\sX_{\pi,1}).\] 
\end{lem}
\begin{proof}The cocycle condition easily follows from the integrability of the Higgs field. We show that the local connections coincide on the overlaps, that is
\[\left(G_{ij}\otimes \id\right)
\circ
 \left(\rmd + \frac{\rmd \Phi_i}{p}(\Phi_i^*\tnabla)\right)
=\left(\rmd + \frac{\rmd \Phi_j}{p}(\Phi_j^*\tnabla)\right) \circ G_{ij}. \]
It suffices to show
\[\frac{\rmd \Phi_i}{p}(\Phi_i^*\tnabla)
=
G^{-1}_{ij}\circ \rmd G_{ij}
+
G^{-1}_{ij}\circ \frac{\rmd \Phi_j}{p}(\Phi_j^*\tnabla) \circ G_{ij}.\]
Since $G^{-1}_{ij}\circ \rmd G_{ij}=\rmd h_{\Phi_i,\Phi_j}(\overline{\Phi}^*\tnabla)$ and $G_{ij}$ commutes with $\frac{\rmd \Phi_j}{p}(\Phi_j^*\tnabla)$ we have
\begin{equation*}
\begin{split}
G^{-1}_{ij}\circ \rmd G_{ij}+G^{-1}_{ij}\circ \frac{\rmd \Phi_j}{p}(\Phi_j^*\tnabla) \circ G_{ij} 
& =\rmd h_{\Phi_i,\Phi_j}(\overline{\Phi}^*\tnabla)
+\frac{\rmd \Phi_j}{p}(\Phi_j^*\tnabla)\\
& =\frac{\rmd \Phi_i}{p}(\Phi_i^*\tnabla)
\end{split}
\end{equation*}
by the integrability of the Higgs field.
Thus we glue those local bundles with connections into a global bundle with connection via $G_{ij}$.
\end{proof}
\begin{lem} To give an object in the category $\MF(\sX_{\pi,1})$ is equivalent to give a tuple $(V,\nabla,\Fil,\phi)$ satisfying
\begin{itemize}
\item[i).] $V$ is filtered local free sheaf over $\sX_{\pi,1}$ with local basis having filtration degrees contained in $[0,a]$;
\item[ii).] $\nabla:V\rightarrow V\otimes_{\mO_{\sX_{\pi,1}}} \Omega^1_{\sX_{\pi,1}}$ is an integrable connection satisfying the Griffiths transversality;
\item[iii).] $ \varphi:C_{\sX_{\pi,1}}^{-1}\widetilde{(V,\nabla,\Fil)}\simeq (V,\nabla)$ is an isomorphism of sheaves with connections over $\sX_{\pi,1}$.
\end{itemize}
\end{lem}

\subsubsection{The functors $\bD$ and $\bD^P$}
For an object in $\MFa(\sX_{\pi,1})$, we get locally constant sheaves on $\mU_K$ by applying the local $\bD$-functors. These locally constant sheaves can be expressed in terms of certain finite \'etale coverings. They can be glued into a finite covering of $\mX_{\pi,K}=\mX_K$. We have the following result.   
\begin{thm}
Suppose that $X$ is a proper smooth and geometrically connected scheme over $W$. Then there exists a fully faithful contravariant functor $\bD$ from $\MFa(\sX_{\pi,1})$ to the category of $\bF_p$-representations of $\pi_1(\mX_K)$. The image of $\bD$ on objects is closed under subobjects and quotients. Locally $\bD$ is given by the same as in Lemma~\ref{lem:defiFunctorD}. 
\end{thm}

Again one can define the category $\MFa(\sX_{\pi,1}^o)$ in the logarithmic case, if one replaces all "connections" by "logarithmic connections" and "Frobenius lifting" by "logarithmic Frobenius lifting". We also have the version of $\MF_{[0,a],f}(\sX_{\pi,1}^o)$ with endomorphism structures of $\bF_{p^f}$, which is similar as the \emph{Variant 2} discussed in section $2$ of\cite{LSZ13a}. And the twisted versions $\TMF_{[0,a],f}(\sX_{\pi,1}^o)$ can also be defined on $\sX_{\pi,1}$ in a similar way as in \cite{SYZ17}. More precisely, let $\mL$ be a line bundle over $\sX_{\pi,1}$. The $\mL$-twisted Fontaine-Faltings module is defined as follows.
\begin{defi}
An $\mL$-twisted Fontaine-Faltings module over $\sX_{\pi,1}$ with endomorphism structure is a tuple \[((V,\nabla,\Fil)_0,(V,\nabla,\Fil)_1,\cdots,(V,\nabla,\Fil)_{f-1},\varphi_\cdot)\]
where $(V,\nabla,\Fil)_i$ are objects in $\MCF(\sX_{\pi,1}^o)$ equipped with isomorphisms in $\MIC(\sX_{\pi,1}^o)$
\[\varphi_i:C_{\sX_{\pi,1}}^{-1}\widetilde{(V,\nabla,\Fil)}_i\simeq (V,\nabla)_{i+1} \text{ for } i=0,1,\cdots,f-2;\]
and 
\[\varphi_{f-1}:C_{\sX_{\pi,1}}^{-1}\widetilde{(V,\nabla,\Fil)}_{f-1} \otimes (\mL^{p},\nabla_\can)\simeq (V,\nabla)_0.\] 
\end{defi}
 The proof of Theorem $0.4$ in \cite{SYZ17} works in this context. Thus we obtain the following result.
\begin{thm}\label{thm:functorD^P}
Suppose that $\mX$ is a proper smooth and geometrically connected scheme over $W$ equipped with a smooth log structure $\mD/W(k)$. Suppose that the residue field $k$ contains $\bF_{p^f}$. Then there exists an exact and fully faithful contravariant functor $\bD^P$ from $\TMF_{a,f}(\sX_{\pi,1}^o)$ to the category of projective $\bF_{p^f}$-representations of $\pi_1(\mX_{K}^o)$. The image of $\bD^p$  is closed under subobjects and quotients.
\end{thm}

Recall that $\{\sU_i\}_i$ is an open covering of $\sX$. A line bundle on $\sX$ can be expressed by the transition functions on $\sU_{ij}$.
\begin{lem}
Let $\mL$ be a line bundle on $\sX_{\pi,1}$ expressed by $(g_{ij})$. Denote by $\widetilde{\mL}$ the line bundle on $\tsX_{\pi,1}$ defined by the same transition functions $(g_{ij})$. Then one has 
\[C^{-1}_{\sX_{\pi,1}}(\widetilde{\mL},0)=\mL^p.\]
\end{lem}
\begin{proof}
Since $g_{ij}$ is an element in $\mB_{R_{ij}} \subset \tmB_{R_{ij}}$, by diagram~(\ref{diag:FrobLift}), one has
\[\Phi(g_{ij})\equiv g_{ij}^p \pmod{p}.\] 
On the other hand, since the $p$-connection is trivial, one has 
 \[C^{-1}_{\sX_{\pi,1}}(\widetilde{\mL},0)=(\Phi\,\mathrm{mod}\,p)^*(\widetilde{\mL}).\]
Thus one has $C^{-1}_{\sX_{\pi,1}}(\widetilde{\mL},0)=(\mO_{\widetilde{\sU}_{i,1}},g_{ij}^p)=\mL^p$.
\end{proof}

In a similar way, one can define the Higgs-de Rham flow on $\sX_{\pi,1}$ as a sequence consisting of infinitely many alternating terms of Higgs bundles over $\tsX_{\pi,1}$ and filtered de Rham bundles over $\sX_{\pi,1}$
	\[\left\{ (E,\theta)_{0},  
	(V,\nabla,\Fil)_{0},
	(E,\theta)_{1},  
	(V,\nabla,\Fil)_{1},
	\cdots\right\}\]
with $(V,\nabla)_i=C^{-1}_{\sX_{\pi,1}}((E,\theta)_i)$ and $(E,\theta)_{i+1}=\widetilde{(V,\nabla,\Fil)_i}$ for all $i\geq 0$.

 \emph{$f$-periodic  $\mL$-twisted Higgs-de Rham flow} over $\sX_{\pi,1}$ of level in $[0,a]$ is a Higgs-de Rham flow over $\sX_{\pi,1}$
	\[\left\{ (E,\theta)_{0},  
	(V,\nabla,\Fil)_{0},
	(E,\theta)_{1},  
	(V,\nabla,\Fil)_{1},
	\cdots\right\}\]
equipped with isomorphisms $\phi_{f+i}:(E,\theta)_{f+i}\otimes (\widetilde{\mL}^{p^i},0)\rightarrow (E,\theta)_i$ of Higgs bundles  for all $i\geq0$ 
\begin{equation*}
\tiny\xymatrix@W=2cm@C=-13mm{
	&\left(V,\nabla,\Fil\right)_{0} \ar[dr]^{\widetilde{(\cdot)}}
	&&\left(V,\nabla,\Fil\right)_{1}\ar[dr]^{\widetilde{(\cdot)}}
	&&\cdots \ar[dr]^{\widetilde{(\cdot)}}
	&&\left(V,\nabla,\Fil\right)_{f}\ar[dr]^{\widetilde{(\cdot)}}  
	&&\left(V,\nabla,\Fil\right)_{f+1}\ar[dr]^{\widetilde{(\cdot)}} 
	&&\cdots\\
	\left(E,\theta\right)_{0}\ar[ur]_{\mathcal C^{-1}_{\sX_{\pi,1}}}
	&&\left(E,\theta\right)_{1}\ar[ur]_{\mathcal C^{-1}_{\sX_{\pi,1}}}
	&& \cdots \ar[ur]_{\mathcal C^{-1}_{\sX_{\pi,1}}}
	&&\left(E,\theta\right)_{f}\ar[ur]_{\mathcal C^{-1}_{\sX_{\pi,1}}} \ar@/^20pt/[llllll]^{\phi_f}|(0.33)\hole
	&&\left(E,\theta\right)_{f+1}\ar[ur]_{\mathcal C^{-1}_{\sX_{\pi,1}}} \ar@/^20pt/[llllll]^{\phi_{f+1}}|(0.33)\hole
	&& \cdots \ar@/^20pt/[llllll]^{\cdots}\ar[ur]_{\mathcal C^{-1}_{\sX_{\pi,1}}}\\} 
\end{equation*} 
And for any $i\geq 0$ the isomorphism
\begin{equation*}
 C^{-1}_{\sX_{\pi,1}}(\phi_{f+i}): (V,\nabla)_{f+i}\otimes (\mL^{p^{i+1}},\nabla_{\mathrm{can}})\rightarrow (V,\nabla)_{i},
\end{equation*} 
		strictly respects filtrations $\Fil_{f+i}$ and $\Fil_{i}$. Those $\phi_{f+i}$'s are related to each other by formula
\[\phi_{f+i+1}=\mathrm{Gr}\circ C^{-1}_{\sX_{\pi,1}}(\phi_{f+i}).\]
Just taking the same construction as in \cite{SYZ17}, we obtain the following result.
\begin{thm}\label{thm:equivalent}
There exists an equivalent functor $\IC_{\sX_{\pi,1}}$ from the category of twisted periodic Higgs-de Rham flows over $\sX_{\pi,1}$ to the category of twisted Fontaine-Faltings modules over $\sX_{\pi,1}$ with a commutative diagram 
\begin{equation}
\xymatrix{
\THDF(X_1) \ar[r]^{\IC_{X_1 }}  \ar[d]_{-\otimes_{\mO_{X_{1}}}\mO_{\sX_{\pi,1}}} & \TMF(X_1)  \ar[d]^{-\otimes_{\mO_{X_{1}}}\mO_{\sX_{\pi,1}} }\\
\THDF(\sX_{\pi,1}) \ar[r]^{\IC_{\sX_{\pi,1}}} & \TMF(\sX_{\pi,1})\ .\\
}
\end{equation}
\end{thm}

\subsection{degree and slope}
Recall that $\sX_\pi$ is a smooth formal scheme over $\mB_{W_\pi}$. Then $\sX_{\pi,1}$ and $X_1$ are the modulo-$p$ reductions of $\sX_\pi$ and $X$ respectively. Also note that $X_1$ is the closed fiber of $Y_1=\mX_\pi\otimes_{\bZ_p}\bF_p$, $\sX_{\pi,1}=\sX_\pi\otimes_{\bZ_p}\bF_p$, $\mX_\pi$ and $\sX_\pi$.
\begin{equation}
\xymatrix@R=2mm{
X_1\ar[r] \ar[dd] &\mX_{\pi,1} \ar[rr] \ar[dd] \ar[dr] && \sX_{\pi,1} \ar[dd]|(0.5)\hole \ar[dr]&\\
&& \widetilde{\sX_\pi}_1  \ar[rr] \ar[dd] && X_1 \ar[dd]\\
X_1\ar[r] \ar[dd]  & \mX_\pi \ar[dr] \ar[dd]  \ar[rr]|(0.5)\hole && \sX_\pi \ar[dr]\ar[dd]|(0.5)\hole &\\
&& \widetilde{\sX_\pi}  \ar[rr] \ar[dd]  && \mX \ar[dd]  \\
\Spec k \ar[r] & \Spec W_\pi \ar[drrr]|(0.34)\hole \ar[rr]|(0.5)\hole
 \ar[dr] && \Spec \mB_{W_\pi} \ar[dr] &\\
&& \Spec\widetilde{\mB}_\pi \ar[rr] && \Spec W\\
}
\end{equation}
For a line bundle $V$ on $\sX_{\pi,1}$ (resp. $\tsX_{\pi,1}$), $V\otimes_{\mO_{\sX_{\pi,1}}}\mO_{X_1}$ (resp. $V\otimes_{\mO_{\tsX_{\pi,1}}}\mO_{X_k}$) forms a line bundle on the special fiber $X_1$ of $\mX$. We denote
\[\deg(V):=\deg(V\otimes_{\mO_{\sX_{\pi,1}}}\mO_{X_1}).\]
For any bundle $V$ on $\sX_{\pi,1}$ (resp. $\tsX_{\pi,1}$) of rank $r>1$, we denote
\[\deg(V):=\deg(\bigwedge_{i=1}^{r}V).\]

By Lemma~\ref{lem:FrobLift}, the modulo-$p$ reduction of the Frobenius lifting is globally well-defined. We denote it by $\Phi_1:\tsX_{\pi,1}\rightarrow \sX_{\pi,1}$. Since $\tsX_{\pi,1}$ and $\sX_{\pi,1}$ have the same closed subset $X_1$, we have the following diagram 
\begin{equation}
\xymatrix{
X_1 \ar[r]^{\widetilde{\tau}} \ar[d]^{\Phi_{X_1}} & \tsX_{\pi,1} \ar[d]^{\Phi_1}\\
X_1 \ar[r]^{\tau} & \sX_{\pi,1}\\
}
\end{equation}
Here $\tau$ and $\widetilde{\tau}$ are closed embeddings and $\Phi_{X_1}$ is the absolute Frobenius lifting on $X_1$. We should remark that the diagram above is not commutative, because $\Phi_1$ does not preserve the defining ideal of $X_1$.

\begin{lem}\label{lem:IsoPullbacks} Let $(V,\nabla,\Fil)$ be an object in $\MCF(\sX_{\pi,1})$ of rank $1$. Then there is an isomorphism 
\[\Phi_{X_1}^*\circ\widetilde{\tau}^*(\tV)\overset{\sim}{\longrightarrow} \tau^*\circ\Phi_1^*(\tV).\]
\end{lem}
\begin{proof}Recall that $\{\sU_i\}_i$ is an open covering of $\sX$. We express the line bundle $V$ by the transition functions $(g_{ij})$, where $g_{ij}\in \left(\mB_{R_{ij}}/p\mB_{R_{ij}}\right)^\times$. Since $V$ is of rank $1$, the filtration $\Fil$ is trivial. Then 
by definition $\tV$ can also be expressed by $(g_{ij})$. Since $g_{ij}\in\mB_{R_{ij}}/p\mB_{R_{ij}}$, one has
\[(\Phi_{X_1}\mid_{U_{i,1}})^*\circ(\widetilde{\tau}\mid_{\widetilde{\sU}_{i,1}})^*(g_{ij}) = (\tau\mid_{\sU_{i,1}})^*\circ(\Phi_1\mid_{\sU_{i,1}})^*(g_{ij}),\]
by diagram~(\ref{diag:FrobLift}). This gives us the isomorphism $\Phi_{X_1}^*\circ\widetilde{\tau}^*(\tV)\overset{\sim}{\longrightarrow} \tau^*\circ\Phi_1^*(\tV)$.
\end{proof}

\begin{lem}\label{lem:deg&C^-1} Let $(V,\nabla,\Fil)$ be an object in $\MCF(\sX_{\pi,1})$. Then we have
\[\deg(\tV)=\deg(V) \text{ and } \deg(C^{-1}_{\sX_{\pi,1}}(\tV))=p\deg(\tV).\]
\end{lem}
\begin{proof}
Since the tilde functor and inverse Cartier functor preserve the wedge product and the degree of a bundle is defined to be that of its determinant, we only need to consider the rank $1$ case. 
Now let $(V,\nabla,\Fil)$ be of rank $1$. The reductions of $V$ and $\tV$ on the closed fiber $X_1$ are the same, by the proof of Lemma~\ref{lem:IsoPullbacks}. Then we have 
\[\deg(\tV)=\deg(V).\]
Since the filtration is trivial, the $p$-connection $\tnabla$ is also trivial. In this case, the transition functions $G_{ij}$ in Lemma~\ref{gluingFlatBundle} are idenities. Thus 
\[C^{-1}_{\sX_{\pi,1}}(\tV)=\Phi_1^*(\tV).\]
Recall that $\deg(\Phi_1^*(\tV))=\deg(\tau^*\circ\Phi_1^*(\tV))$ and $\deg(\tV)=\deg(\widetilde{\tau}^*(\tV))$. Lemma~\ref{lem:IsoPullbacks} implies $\deg(\tau^*\circ\Phi_1^*(\tV))=\deg(\Phi_{X_1}^*\circ\widetilde{\tau}^*(\tV))$. Since $\Phi_{X_1}$ is the absolute Frobenius, one has $\deg(\Phi_{X_1}^*\circ\widetilde{\tau}^*(\tV))=p\deg(\widetilde{\tau}^*(\tV))$. Composing above equalities, we get $\deg(C^{-1}_{\sX_{\pi,1}}(\tV))=p\deg(\tV)$.
\end{proof}

\begin{thm}
Let	$\mE=\left\{ (E,\theta)_{0},  
(V,\nabla,\Fil)_{0},
(E,\theta)_{1},  
(V,\nabla,\Fil)_{1},
\cdots\right\}$ be an $L$-twisted $f$-periodic Higgs-de Rham flow with endomorphism structure and log structure over $X_{1}$. Suppose that the degree and rank of the initial term $E_0$ are coprime. Then the projective representation $\bD^P\circ\IC_{\sX_{\pi,1}}(\mE)$ of $\pi_1(X_{K_0}^o)$ is still irreducible after restricting to the geometric fundamental group $\pi_1(X^o_{\overline{K}_0})$, where $K_0=W[\frac1p]$.
\end{thm}
\begin{proof}
Let $\rho:\pi_1(X_{K_0}^o)\rightarrow \mathrm{PGL}(\bD^P\circ\IC_{\sX_{\pi,1}}(\mE))$ be the projective representation. Fix a $K_0$-point in $X_{K_0}$, which induces a section $s$ of the surjective map $\pi_1(X_{K_0}^o)\rightarrow \mathrm{Gal}(\overline{K}_0/K_0)$. We restrict $\rho$ on $\mathrm{Gal}(\overline{K}_0/K_0)$ by this section $s$, whose image is finite. And there is a finite field extension $K/K_0$ such that the restriction of $\rho$ by $s$ on $\mathrm{Gal}(\overline{K}_0/K)$ is trivial. Thus 
\[\rho(\pi_1(X_{K}^o))=\rho(\pi_1(X^o_{\overline{K}_0})).\]
It is sufficient to show that the restriction of $\rho$ on $\pi_1(X_{K}^o)$ is irreducible. Suppose that the restriction of $\bD^P\circ\IC_{X_1}(\mE)$  on $\pi_1(X_{K}^o)$ is not irreducible.
Since the functors $\bD^P$ and $C^{-1}_{\sX_{\pi,1}}$ are compatible with those over $X_1$, the projective representation 
$\bD^P\circ \IC_{\sX_{\pi,1}}(\mE\otimes_{\mO_{X_1}}\mO_{\sX_{\pi,1}})=\bD^P\circ\IC_{X_1}(\mE)$ is also not irreducible. Thus there exists a non-trivial quotient, which is the image of some nontrivial sub $\mL=L\otimes_{\mO_{X_1}}\mO_{\sX_{\pi,1}}$-twisted $f$-periodic Higgs-de Rham flow of $\mE\otimes_{\mO_{X_1}}\mO_{\sX_{\pi,1}}$
\[\left\{ (E',\theta')_{0},  
(V',\nabla',\Fil')_{0},
(E',\theta')_{1},  
(V',\nabla',\Fil')_{1},
\cdots\right\},\] 
 under the functor $\bD^P\circ\IC_{\sX_{\pi,1}}$ according Theorem~\ref{thm:functorD^P} and Theorem~\ref{thm:equivalent}. 
Since $E'_0$ is a sub bundle of $E_0\otimes_{\mO_{X_1}}\mO_{\sX_{\pi,1}}$, we have $1\leq \mathrm{rank}(E'_0) <\mathrm{rank}(E_0)$. 
By Theorem~4.17 in~\cite{OgVo07}, $\deg(E_{i+1})=p\deg(E_i) \text{ for }i\geq 0$
and 
$\deg(E_{0})=p\deg(E_{f-1})+\mathrm{rank}(E_{0})\times \deg(L)$.
Thus 
\begin{equation}
\frac{\deg (E_0)}{\mathrm{rank} (E_0)}=\frac{\deg (L)}{1-p^f}.
\end{equation}
Similarly, by Lemma~\ref{lem:deg&C^-1}, one gets 
\[\frac{\deg (E'_0)}{\mathrm{rank} (E'_0)}=\frac{\deg (\mL)}{1-p^f}.\]
Since $\deg(\mL)=\deg(L)$, one has $\deg(E_0)\cdot\mathrm{rank}(E_0')=\deg(E_0')\cdot\mathrm{rank}(E_0)$.
Since $\deg (E_0)$ and $\mathrm{rank} (E_0)$ are coprime, $\mathrm{rank}(E_0)\mid \mathrm{rank}(E'_0)$. This contradicts to $1\leq \mathrm{rank}(E'_0) <\mathrm{rank}(E_0)$. Thus the projective representation $\bD^P\circ\IC_{X_1}(\mE)$ is irreducible.
\end{proof}

\end{document}